\documentclass[10pt]{article}
\usepackage{amssymb,amsmath,makeidx,amscd,verbatim,amsthm,graphicx,mathrsfs}
\usepackage{lscape}
\usepackage{color}
\newtheorem{thm}{Theorem}[section]
\newtheorem{prop}[thm]{Proposition}
\newtheorem{lem}[thm]{Lemma}

\theoremstyle{definition}
\newtheorem{defx}[thm]{Definition}
\newtheorem{rem}[thm]{Remark}
\newtheorem{ex}[thm]{Example}

\numberwithin{equation}{section}

\newcommand{\be}{\begin{enumerate}}
	\newcommand{\ee}{\end{enumerate}}
\newcommand{\bq}{\begin{eqnarray*}}
	\newcommand{\eq}{\end{eqnarray*}}

\begin{document}
	\pagenumbering{arabic} \baselineskip 10pt
	\newcommand{\disp}{\displaystyle}
	\renewcommand*\contentsname{Table of Contents}
	\thispagestyle{empty}
	\newcommand{\HRule}{\rule{\linewidth}{0.1mm}}
	\linespread{1.0}
	\pagenumbering{arabic} \baselineskip 10pt
	\thispagestyle{empty}
	\title{\textbf{On Dirichlet Spaces of Homogeneous Type Via Heat Kernel}}
	\author{J. I. Opadara$^1$ and M. E. Egwe$^2$ \\ Department of Mathematics, University of Ibadan, Ibadan.\\ $^1$\emph{jossylib@gmail.com}\; $^2$\emph{murphy.egwe@ui.edu.ng}}
	\maketitle
	\large	
	\begin{abstract}
This paper considers the properties of Dirichlet Spaces of Homogeneous type which consist of band limited functions that are nearly
exponential localizations on $\mathbb{R}^k.$ This is a powerful tool in harmonic analysis and it makes various spaces of functions and distributions more approachable, utilizable and providing non-zero representation of natural function spaces, such as Besov space,
on $\mathbb{R}^k$. Spheres and homogeneous spaces can also admit such frames on the intervals and balls. Here, we present mainly the band limited frames that are well-localized in the general setting of Dirichlet spaces of Homogeneous type which have doubling measure and a local scale-invariant Poincare inequality which generates heat kernels through the Gaussian bounds and H$\ddot{o}$lder's continuity. As an application of this build-up, band limited frames are generated in the context of Lie groups which are homogeneous in nature with polynomial volume growth, complete Riemannian manifolds with Ricci curvature bounded from below and admits the volume doubling property, together with other settings. In this general setting, decomposition of Besov spaces was done with the new frames.
\ \\
\ \\
\textbf{Keywords:} Dirichlet Spaces; Ricci curvature; Lie Algebra; Besov space; Poincare inequality.
\textbf{Mathematics Subject Classification (2020):} 30H25, 53C44, 16W25, 53B21.
	\end{abstract}
	
	\pagenumbering{arabic} \baselineskip 10pt
	\section{Introductions}
	Decomposition systems, bases or frames, consisting of band limited functions of nearly exponential space localization have played vital role in the theoretical and computational Harmonic Analysis, Partial Differential Equations, Statistics, Approximation theory and their applications. Meyer's wavelets and the frames, the $\phi$-transform, of Frazier and Jawerth are the thrusting examples of such decomposition systems which play a significant role in the solution of numerous theoretical and computational problems.  Frames of similar nature have been developed in non-standard settings such as, on the sphere and more general Homogeneous spaces, on the interval and ball with weights, and used extensively in statistical applications.\\
	The main goal of this paper is to extend  the construction of band limited frames with elements of nearly exponential space, localization to the general setting of strictly local regular Dirichlet space, of Homogeneous type, with doubling measure and local scale-invariant Poincar$\acute{e}$ inequality that leads to a Markovian heat kernel with Gaussian bounds and H$\ddot{o}$lder continuity. New light will be shed on the existing frames and decomposition of spaces, and also develop band limited localized frames in the context of Homogeneous Lie Groups with polynomial volume growth.
	\section{Preliminaries}
	Throughout this paper,  we shall represent by $G$, the homogeneous Lie group and $\mathfrak{g}$ its Lie algebra.
	\begin{defx}
		\normalfont
		\ \\
		A family of dilations of a Lie algebra $\mathfrak{g}$ is a family of linear mappings
		\begin{equation*}
			\{E_{t},t>0\}
		\end{equation*}
		from $\mathfrak{g}$ to itself which satisfies the following:\\
		\begin{equation*}
			E_{t}=exp(B\ln t)=\sum_{\jmath=0}^{\infty}\frac{1}{\jmath!}(\ln (t)B)^{\jmath},
		\end{equation*}
where $B$ is a diagonalisable linear operator on $\mathfrak{g}$ which has positive eigenvalues, $\ln (t)$ stands for natural logarithm of $t>0$, each $E_{t}$ forms a morphism of the Lie algebra $\mathfrak{g}$, i.e.,  a linear mapping from $\mathfrak{g}$ to itself with respects to the Lie brackets:
		\begin{equation*}
			[E_{t}U,E_{t}V]=E_{t}[U,V],\hspace{1cm}\forall \;\; X,Y\in \mathfrak{g},\;\; t>0.
		\end{equation*}
	\end{defx}
	\begin{defx}
		\normalfont
		\ \\
		A homogeneous Lie group is a connected simply connected Lie group whose Lie algebra admits dilations.
	\end{defx}
	\begin{defx}
		\normalfont
		\ \\
Let $\textbf{S}$ be a a nonempty set which satisfies a quasi-distance $a(l,p)$ and suppose that a positive measure $\sigma$ is defined on a $\sigma$-algebra of subsets of $\textbf{S}$ containing the $a$-open subsets and the open ball $B(s,r)$ is given, which satisfies that two finite constants, $c>1$ and $K>0$, such that
$$0<\sigma(B(s,cr))\leq K \cdot \sigma(B(s,r))<\infty$$
for every $s\in \textbf{S}$ and $r>0$.\\
A set $\textbf{S}$ with a quasi-distance $a(l,p)$ and a measure $\sigma$ satisfying the conditions above will be called a space of Homogeneous type and this shall be denoted by $(\textbf{S},a,\sigma)$.\\
We say that a space of homogeneous type is normal if we can find two non-negative finite constants $k_1>0$ and $k_2>0$ for which $$k_1r\leq \sigma (B(s,r))\leq k_2r$$ holds, for all $s\in \textbf{S}$ and $r$, $\sigma(\{s\})\leq r\leq \sigma (s)$.
	\end{defx}
	\begin{thm}
		\normalfont
		(\cite{Coulhon1},\cite{Coulhon2})
		\ \\
		Let $(\textbf{S},a,\sigma)$ be a space of homogeneous type. Given a function $\phi (s)$, which is integrable on bounded subsets, and a ball $B(s,r)$ we shall denote by $m_B(\phi)$ the mean value of $\phi (s)$ on $B(s,r)$, which is defined by $$m_B(\phi)=\sigma(B(s,r))^{-1} \int_B \phi(s)d\sigma(s).$$
	\end{thm}
	\begin{defx}
		\normalfont
Let $\mathfrak{M}$  a smooth manifold and $\mathcal{B}(\mathfrak{M},\mathfrak{M})$, be the space of bounded linear operators on $\mathfrak{M}.$
A semigroup of operators on $\mathfrak{M}$ is a continuous operator valued function, $R:[0,+\infty]\longrightarrow \mathcal{B}(\mathfrak{M},\mathfrak{M})$, such that $R(t+s)=R(t)R(s)\;\;\; \mbox{for all} \;\; t,s\geq 0$ and $R(0)=id_{\mathfrak{M}}$.\\
The definition means that, if $\{R(t):t\geq0\}$ is a family of operators on $L^2(\mathfrak{M})$. We say that it is a semigroup if
		\begin{equation*}
			R(0)=id_{\mathfrak{M}},\;\;\;\; R(t+s)=R(t)R(s)\;\;\; \forall \;\; t,s\geq0.
		\end{equation*}
	\end{defx}
	
	\begin{defx}
		(a). Let $T$ be a linear operator in Hilbert space, $\mathcal{H}$, then its adjoint $T^{\ast}$ is defined as follows. The domain $D(T^{\ast})$ consists of the vector $u\in\mathcal{H}$ for which the map $$D(T)\ni u \mapsto \langle u,Tv\rangle \in \mathbb{C}\;\; \mbox{i.e.}\;\; T:D(T)\longrightarrow \mathbb{C}$$
		is bounded with respect to the $\mathcal{H}$-norm. For such $u$ there exists, a unique vector denoted by $T^{\ast}u$ (by Riesz Theorem) such that $$\langle u,Tv\rangle = \langle T^{\ast}u,v\rangle \hspace{0.5cm} \forall \hspace{0.2cm} v\in D(T)$$
		(b). A linear operator $T$ in $\mathcal{H}$ is symmetric (Hermitian) if $\langle u,Tv\rangle = \langle Tu,v\rangle \hspace{0.5cm} \forall \hspace{0.2cm} u,v\in D(T).$\\
		(c). The linear operator $T$ in $\mathcal{H}$ is said to be self-adjoint if $T=T^{\ast}$.\\
		(d). A linear operator whose closure $(\bar{T})$ is self-adjoint is called essentially self-adjoint.
\end{defx}

	\begin{defx}
\normalfont
Let $(R_t)_{t\geq 0}$ be a strongly continuous self-adjoint contraction semigroup on $L^2(\mathfrak{M},\sigma)$. The semigroup $(R_t)_{t\geq 0}$ is called Markovian if and only if for every $f\in L^2(\mathfrak{M},\sigma)$ and $t>0$\\
		(a) $f\geq 0$, almost everywhere $\Rightarrow R_tf\geq 0$, almost everywhere\\
		(b) $f\leq 1$, almost everywhere $\Rightarrow R_tf\leq 1$, almost everywhere
	\end{defx}
	\begin{defx}
\normalfont
Let $(\omega,D(\omega))$ be a densely defined closed symmetric form on $L^2(\mathfrak{M},\sigma)$. The form $\omega$ is called a Dirichlet form if it is Markovian, that is, if it has the property that if $u\in D(\omega)$ and $v$ is a normal contraction of $u$ then $v\in D(\omega)$ and $\omega(v,v)\leq \omega(u,u)$.
	\end{defx}
	\begin{defx}
		\normalfont
		\ \\
		The Dirichlet form is locally strong if for any two functions $f,g\in D(\omega)$ with compact supports, such that $f$ is constant in a neighbourhood of the support of $g$, then we have $\omega(f,g)=0$.
		
	\end{defx}

	\begin{rem}
		\normalfont
A Dirichlet metric (or intrinsic metric), $a\omega$, can be associated with a strongly local Dirichlet form. It measures the distance between points in the space based on the Dirichlet form.\\
		The metric is separable if the underlying metric space has a countable dense subset. \\
		A strongly local Dirichlet space is strictly local, if the Dirichlet metric, $a\omega$, is a metric on the space, and the topology it induces on the space is the same as the original topology of the space.
	\end{rem}
We recall the following general properties of operators from classical functional Analysis.	
\begin{prop}(\cite{Yoshida})
\normalfont
Symmetric operators are closeable.
\end{prop}
	
\begin{prop}(\cite{Yoshida})
\normalfont
Let $T$ be an injective self-adjoint operator, then its inverse is also self-adjoint.
\end{prop}
\begin{prop}(\cite{Yoshida})
\normalfont
Let $(e_n)$ be an orthogonal basis of a Hilbert space $\mathcal{H}$ and$(\lambda_n)\subset\mathbb{R}$ and let $T$ be a linear operator on $\mathcal{H}$ such that for any $n$ one has $e_n\in D(T)$ and $Te_n=\lambda_n e_n$, then $T$ is essentially self-adjoint.
	\end{prop}
	\noindent
Now we are set to construct the Dirichlet space needed for this setting. We assume that $L$ is a positive symmetric operator on $L^2 (\mathfrak{M},\sigma)$ with domain $D(L)$, dense in $ L^2 (\mathfrak{M},\sigma)$.
	Now, we write $ L^p := L^p (M,\sigma)$ in the sequel. And $L$ can be associated with a non-negative  symmetric form.
	$$ \omega (f,g) = 	(Lf,g)	= \omega (g,f),\;\; \mbox{where}\;\; \omega(f,f) = 	(Lf,f)	\geq 0,$$
	where the domain $ D(\omega) = D(L)$. We consider the prehilbertian structure on $D(\omega)$ induced by
	$$ \lVert f \rVert_{\omega} ^2 = \lVert f \rVert_{2} ^2 + \omega(f,f).$$ This is not complete in a general setting but is closed in $L^2$. We call  $\overline{\omega}$ and $D(\overline{\omega})$ the closure of ${\omega}$ and of its domain. And this give rise to a self-adjoint extension $\overline{L}$ of $L$ with domain $D(\overline{L})$ which consists of all $f \in D(\overline{\omega})$ for which there exists $v \in L^2$ such that $ \overline{\omega}(f,g) =	\langle v, g \rangle$
	$\forall$ $ g \in D(\overline{\omega})$ and $ \overline{L}f = v$.
	For this fact, we say that $\overline{L}$ is positive and self-adjoint, and that
	$$ D(\overline{\omega}) = D(\overline{L})^{1/2}, \hspace{2 cm} \overline{\omega}(f,g) = \langle (\overline{L})^{1/2} f, (\overline{L})^{1/2} g \rangle.$$
	If we use the classical spectral theory of positive self-adjoint operators, we can now associate with $ \overline{L}$ a self-adjoint strongly continuous contraction semigroup $ R_t = e^{-t\overline{L}} $ on $ L^2 (\mathfrak{M},\sigma)$. Thus,
	$$ R_t=e^{-t\overline{L}} = \int_{0}^{\infty} e^{-\beta t} d S_{\beta},$$
	where $S_{\beta}$ stands for spectral resolution that associates with $\overline{L}$. Furthermore, this semigroup has
	a holomorphic extension to the complex half-plane $ \textbf{Re} z > 0$.\\
Next, let us assume that $ R_t$  is a submarkovian semigroup: Let $ 0 \leq f \leq 1 $ and $ f \in L^2 $ and this implies that  $ 0 \leq R_t f \leq 1 $. Then $R_t$ can be straighten as a contraction operator
	on $ L^p$, $ p \in [1, \infty] $, and this preserves the positive condition, obeying $R_t 1 \leq 1 $, hence it yields a strongly continuous contraction semigroup on $ L^p$, for $ p \in [1, \infty] $. One can verify a sufficient condition for this on $D(L)$, that is,for every $ \epsilon >0 $ $\exists$ $ \varphi_{\epsilon} : \mathbb{R} \longrightarrow [- \epsilon, 1 + \epsilon]$ such that $ \varphi_{\epsilon}$ is non-decreasing,
$\varphi_{\epsilon} (t) = t $ 	for $ t \in [0,1] $ and $  \varphi_{\epsilon} (f) \in D(\overline{\omega})$ and
 $$\overline{\omega}( \varphi_{\epsilon}(f),  \varphi_{\epsilon}(f)) \leq \omega(f,f), \hspace{1 cm} \forall f \in D(L).$$
	Based on the facts above, Dirichlet space can be written as \textbf{$(\overline{\omega},D(\overline{\omega}))$}  where $D(\overline{\omega}) \cap L^{\infty} $ is an algebra.\\
	Now, let us make an assumption that $ \overline{\omega}$ is locally strong, i.e. for $ f,g \in D(\overline{\omega})$, $ \overline{\omega}(f,g) = 0$  whenever $f$ has a compact support and $g$ being a constant on a neighbourhood of the support of $f$. We also make an assumption that $ \overline{\omega}$ is regular, this means that the space $ K_a (\mathfrak{M})$ form continuous functions on $\mathfrak{M}$ with compact support has the attribute that the algebra $ K_a (\mathfrak{M}) \cap D(\overline{\omega})$ is dense in $ K_a (\mathfrak{M})$ with respect to the supremum  norm, and dense in $D(\overline{\omega})$ in
	the norm $ \left({\overline{\omega}(f,f) + \displaystyle \lVert f \rVert_2 ^2}\right)^{1/2}$.\\
	At this juncture, let us give a necessary and sufficient condition for $\overline{\omega}$ to be locally strong  and regularity for $ D(L)$:$ \overline{\omega}$ is locally strong and regular if\\
	(a) $ D(L)$ is a subalgebra of $ K_a (\mathfrak{M}).$ i.e.,   $0 = \omega (f,g) = \langle Lf,g \rangle$ if $ f,g \in D(L)$, where $f$ has a compact support and $g$ is the constant on a neighbourhood of the support of $f$, and\\
	(b) for any open set $V$ and compact set $C$ such that $ C \subset V $, there exists $ v \in D(L)$ with $ v \geq 0$, supp $ v \subset V$, and $ v \equiv 1$ on $C$ (hence, $D(L)$ forms a dense subalgebra of $ K_a (\mathfrak{M})$ as well as dense in $D(\overline{\omega})$).\\
	From the assumptions above, it can be seen that there exists a bilinear symmetric form $ d\varGamma $ defined on $ D(\overline{\omega}) \times D(\overline{\omega})$ having values in the signed Radon measures on $\mathfrak{M}$ such that
	$$ \omega (\psi f,g) + \omega (f, \psi g) - \omega (\psi, fg) = 2 \int_{\mathfrak{M}} \psi d \varGamma (f,g) \hspace{0.2 cm} \mbox{for} \hspace{0.2 cm} f,g,\psi K_a (\mathfrak{M}) \cap D (\overline{\omega}),$$
	which can be tested for $ \overline{\omega} (f,g) = \int_{\mathfrak{M}} d  \varGamma (f,g)$ and $ 0 \leq d \varGamma (f,f)$. \\
	As a matter of fact, if $D(L)$ behaves like $ K_a (\mathfrak{M})$, then $d \varGamma$ will be absolutely continuous with regard to $ \sigma $, and
	$$ d \varGamma(f,g)(v) = \varGamma (f,g)(v)d \sigma (v),$$
	$$ \varGamma(f,g) = \frac{1}{2} (L(fg)-fLg - gLf).$$
	On the other hand, $ \overline{\omega}$ admits a “carré du champ” (Bouleau and Hirsch, 1991): There exists a bilinear map $\varGamma: D(\overline{\omega}) \times D(\overline{\omega}) \longrightarrow L^1 : f,g \mapsto \varGamma (f,g)$ such that $ 0\leq\varGamma (f,f) (v)$,
	$$\overline{\omega} (\psi f,g) + \overline{\omega} (f, \psi g) - \overline{\omega} (\psi, fg)
	= 2 \int_{\mathfrak{M}} \psi (v) \varGamma (f,g) (v) d \sigma (v) \hspace{0.3 cm} \forall\;\; f,g,\psi \in D(\overline{\omega}) \cap L^{\infty},$$
	and $ \overline{\omega}(f,g) = 2\int_{\mathfrak{M}} \varGamma (f,g)(v)d\sigma (v)$.\\
	An intrinsic distance on $\mathfrak{M}$ can be defined by
	$$ a(x,y) = sup\{v(x) - v(y) : v\in D(\overline{\omega})\cap K_a (\mathfrak{M}),\; d\varGamma(v,v) = \eta(v)(x)d\sigma(x),\; \eta (v)(x)\geq 1\}.$$
Now assume that $ a : \mathfrak{M} \times \mathfrak{M} \longrightarrow [0,\infty]$ is a true metric, which generates the original topology on $\mathfrak{M}$ and $(\mathfrak{M}, a)$ is a complete metric space.\\
	The implication of this assumption now is that,  $\mathfrak{M}$ is a connected space, and the closed ball $\overline{B(x,r)} := \{ y \in \mathfrak{M},a(x, y) \leq r \}$ (which is compact) is the closure of an open ball $B(x,r)$.\\
	Next we describe a best outline when the needed Gaussian bound, H$\ddot{o}$lder continuity, and Markov property on the heat kernel can be effectively realized. In the basic concept of strictly local regular Dirichlet spaces combine with a complete intrinsic metric, the following are important:\\
\begin{enumerate}
\item[(i)] The heat kernel admits
	\begin{equation}\label{E}
		\frac{c_1 ^{\prime}exp \{ - \frac{c_1 a (x,y)^2}{t}\}}{\left(\sigma(B(x,\sqrt{t}))\sigma(B(y,\sqrt{t}))\right)^{1/2}} \leq h_t (x,y) \leq \frac{c_2 ^{\prime}exp \{ - \frac{c_2 a(x,y)^2}{t}\}}{\left(\sigma(B(x,\sqrt{t}))\sigma(B(y,\sqrt{t}))\right)^{1/2}}
	\end{equation}
	for $x,y\in \mathfrak{M}$ and $1\geq t \geq 0$.
\item[(ii)(a)] $(\mathfrak{M},a,\sigma)$ is a local doubling measure space: That is, there is a constant $k >0$ such that $\sigma (B(x,2r)) \leq 2^k \sigma (B(x,r))$ with $ x \in \mathfrak{M}$ and $1>r>0$.
\item[(b)] It satisfies the local scale-invariant Poincar$\acute{e}$ inequality: That is, if there is a constant $K >0$ such that for any ball $ B = B(x,r)$ with $1\geq r >0$ where $x \in \mathfrak{M}$, and any function $f\in D(\overline{\omega})$,
	$$ \int_B |f-f_B|^2 \leq Kr^2 \int_B d \varGamma(f,f),$$
where $f_B$ is the mean of $f$ over $B$. Moreover, the above property
	is equivalent to a local parabolic Harnack inequality.
\end{enumerate}
Consequently, given a situation which fits into the framework of strictly local regular Dirichlet spaces, with a complete intrinsic metric, it suffices to only verify the local Poincar$\acute{e}$ inequality and the global doubling condition on the measure, then the theory is fully applied.
	\begin{ex}
		\normalfont
		\textbf{Heat Kernel on $[-1, 1]$ Spanned by the Jacobi Operator}\\
		Here we consider the “simple” example of $\mathfrak{M} = [-1, 1]$ with $d \sigma (x) = w_{\gamma, \alpha} (x)kx$, in which $w_{\gamma, \alpha} (x)$ is the classical Jacobi weight defined by:
		$$ w_{\gamma, \alpha} (x) = w (x) = (1-x)^{\gamma} (1+x)^{\alpha}, \hspace{0.1 cm} \gamma, \alpha > -1. $$
		The Jacobi operator is then define as follows\\
		$$Lf(x) = -\frac{[w(x)u(x) f'(x)]^{\prime}}{w(x)} \hspace{0.2 cm} with \hspace{0.2 cm} u(x) := 1-x^2 $$
		and $ D(L) = K^2[-1, 1]$. $LP_i = \beta_i P_i$, where $ P_i (i \geq 0)$ represents the $ith$ degree which normalizes  Jacobi polynomial and $ \beta_i = i(i + \gamma + \alpha + 1)$. Using integration by part we have
		$$ \omega (f,g) := \langle Lf,g \rangle = \int_{-1}^{1} u(x) f'(x) g^{\prime} (x) w_{\gamma, \alpha} (x) dx.$$
	\end{ex}\hfill{$\Box$}
	
	\section{Main Results}
	\begin{defx}
		We say that the metric space $(\mathfrak{M},a_w,\sigma)$ satisfies the volume doubling property, if there exists a constant $K>0$ such that for every $x\in \mathfrak{M}$ and $r>0$ $$\sigma(B(s,2r))\leq K\sigma(B(s,r)).$$
	\end{defx}
	\noindent
	Now let us describe some important tools needed for the development of our theory. Assuming that $(\mathfrak{M},a,\sigma)$ is a space of measurable metric, satisfying the following conditions:\\
	(a) $ (\mathfrak{M},a) $ is a metric space which is locally compact with respect to $ a(\cdotp, \cdotp),$  where $\sigma$ is a positive Radon measure which makes the following volume doubling conditions valid\\
	\begin{equation}\label{A}
		0 < \sigma (B(s,2r)) \leq 2^k \sigma (B(s,r)) < \infty \hspace{0.2 cm}\;\mbox{for}\; x \in \mathfrak{M}\;  \mbox{and}\;  r >0,
	\end{equation}
	where $B(s,r)$ is the open ball with center $s$ and of radius $r$ and $k >0$ is a constant which stands for a dimension of $M$.\\
	(b) We assume that the reverse of the doubling condition is also true, that is, there is a positive constant $ \alpha$ for which
	\begin{equation}\label{F}
		\sigma (B(s,2r)) \geq 2^{\alpha} \sigma (B(s,r)) \;\; \hbox{for}\; \; s \in M \;\; \hbox{and}\; \; \frac{diam M}{3} \geq r >0.
	\end{equation}
	We shall prove in sequel that this estimate is a consequence of the doubling condition \eqref{A} if $M$ is connected.\\
	(c) We shall also stipulate the following non-collapsing condition: That is, there exists a constant $a >0$ such that
	\begin{equation}\label{G}
		\inf_{s \in M} \sigma (B(s,1)) \geq a, \hspace{0.2 cm} \forall \hspace{0.2 cm} s \in M.
	\end{equation}
	The case where $ \sigma (M) < \infty $ will be shown later that the inequality above follows by \eqref{A}. Hence, the case when  $ \sigma (M) = \infty $ is an additional assumption.\\
	Since the paper is only based on the spaces of homogeneous functions, it will make sense that our assumptions are purely local, and to assume only doubling for the balls whose radii are bounded by some constant, in particular, which would considerably enlarge the range of examples. And the assumptions on the heat kernel $h_t(s_1,s_2)$ are local.\\
	The most important assumption is that the local geometry of the space $ (M,d,\sigma) $ is related to an essentially self-adjoint positive operator $T$ on $\mathbb{L}^2 (M, d \sigma)$ such that the associated semigroup $R_t = exp{(-tT)} $ consists of integral operators with heat kernel $ h_t (s_1,s_2)$ satisfying the conditions:\\
	\\
	(d) The Gaussian upper bound:
	\begin{equation}\label{H}
		h_t (s_1,s_2) \leq \frac{K exp \{- \frac{ad (s_1,s_2)^2}{t} \}}{\left(\sigma (B(s_1, \sqrt{t}))\sigma (B(s_2, \sqrt{t})) \right)^{1/2}} \hspace{0.2 cm} \mbox{for} \hspace{0.1 cm} s_1,s_2 \in M, \hspace{0.2 cm} 1>t>0
	\end{equation}
	One can observe that the combination of results  (\cite{Ouhabaz} \cite{Coulhon3} \cite{Coulhon1},\cite{Coulhon2}) gives that this estimate and the doubling condition in \eqref{A} together with the fact that $ e^{-tT}$ is a Holomorphic semigroup on $L^2 (M, d \sigma)$, that is, $ e^{-zT} $ exists where $z \in \mathbb{C}$, $\textbf{Re} \hspace{0.05 cm} z  \geq 0$, mean that $ e^{-zT} $ is an integral operator with kernel $ h_z (s_1,s_2)$ obeying the estimation below: For any
	$ z = t +iv, \hspace{0.2 cm} 1\geq t>0, \hspace{0.2 cm} v \in \mathbb{R}, \hspace{0.2 cm} \mbox{and} \hspace{0.2 cm} s_1,s_2 \in M$,
	\begin{equation}
		\arrowvert	h_z (s_1,s_2) \arrowvert \leq \frac{K exp\{-a \textbf{Re} \frac{d(s_1,s_2)^2}{z}\}}{\left(\sigma (B(s_1, \sqrt{t}))\sigma (B(s_2, \sqrt{t}))\right)^{1/2}} \;\; \mbox{for} \;\; s_1,s_2 \in M, \;\; 1>t>0
	\end{equation}
	(e) H$\ddot{o}$lder continuity: There exists a constant $\gamma >0 $ such that
	\begin{equation}
		\arrowvert	h_t (s_1,s_2) - h_t (s_1,s_2^{\prime})\arrowvert \leq K \left ( \frac{d(s_2,s_2^{\prime})}{t^{1/2}}\right)^{\gamma} \frac{exp \{- \frac{a d(s_1,s_2)^2}{t}\}}{\left(\sigma (B(s_1, \sqrt{t}))\sigma (B(s_2, \sqrt{t}))\right)^{1/2}} \end{equation}
	$\forall \hspace{0.2 cm}s_1,s_2,s_2^{\prime} \in M$ and $1\geq t>0 $, where $ d(s_2,s_2^{\prime}) \leq t^{1/2}$\\
	(f) Markov property:
	\begin{equation}
		\int_{M} h_t (s_1,s_2) d \sigma (s_2) \equiv 1 \;\; \mbox{for} \;\; t>0,
	\end{equation}
	which by analytic continuation implies that,
	\begin{equation}\label{I}
		\int_{M} h_z (s_1,s_2) d \sigma (s_2) \equiv 1 \;\; \mbox{for} \;\; z = t + iv, t>0,
	\end{equation}
	 $a > 0$ and $K$ above represents the structural constants that will have effect on almost all the constant in subsequent discussions.\\
	The paramount results in this article will be deduced from the conditions above.\\
	The estimate \eqref{A}, imply that
	\begin{equation}\label{J}
		|B(s,\beta r)| \leq (2\beta)^k |B(s,r)|, \hspace{0.2 cm} s \in M, \hspace{0.1 cm} \beta >1, \hspace{0.1 cm} r>0
	\end{equation}
	thus, since $B(s_1,r) \subset B(s_2, d(s_2,s_1) + r)$, we have\\
	\begin{equation}\label{K}
		| B(s_1,r)| \leq 2^k \left( 1+ \frac{d(s_1,s_2)}{r} \right)^k | B(s_2,r)|, \hspace{0.2 cm} s_1,s_2 \in M, \hspace{0.1 cm} r>0.
	\end{equation}
	Reversing the estimate in \eqref{J} gives
	\begin{equation}
		|B(s, \beta r)| \geq (\beta /2)^{\alpha} |B(s,r)|, \hspace{0.2 cm} \beta >1, \hspace{0.1 cm} r>0, \hspace{0.1 cm} 0< \beta r < \frac{\mbox{diam} M}{3}.
	\end{equation}
	Now, combining \eqref{G} with \eqref{J} implies
	\begin{equation}\label{Y}
		\inf_{s\in M} |B(s,r)|\geq \hat{a}r^k, \hspace{0.5 cm} 1\geq r >0,
	\end{equation}
	where $ \hat{a}=2^{-k}a$ and $ a>0$ is the constant from \eqref{G}.\\
	It is worthy to note that $ |B(s,r)| $ can be much larger than $ar^k$ as is evidenced by the case of the
	Jacobi operator on $[-1, 1]$.\\
	We now make a claim which shows that \eqref{G} is true automatically when $ \sigma (M) < \infty$.\\
	\begin{prop}
		\normalfont
		\textbf{[\cite{Coulhon1},\cite{Coulhon2}, 2012]}.
		Let $(M,d,\sigma)$ be a measurable metric space which satisfies \eqref{A}. Then the following holds\\
		(a) 	$\sigma (M) < \infty$ if and only if diam$ M < \infty$. Furthermore, if diam$M = L < \infty$, then
		\begin{equation}
			\inf_{s\in M} |B(s,r)| \geq r^k|M|(2L)^{-k}, \hspace{0.2 cm} L \geq r >0.
		\end{equation}
		(b) $ \sigma (\{s\}) > 0$ for some $s \in M$ if and only if $\{s\} = B(s,r)$ for some $r >0$.\\
	\end{prop}
	\begin{prop}
		\normalfont
		The reverse of doubling condition holds, if $M$ is a connected space, that is, if there exists $ \alpha >0$ such that
		$$ |B(s,2r)| \geq 2^{\alpha} |B(s,r)| \hspace{0.2 cm} \mbox{for} \hspace{0.1 cm} s\in M \hspace{0.1 cm} \mbox{and} \hspace{0.1 cm} 0<r< \frac{\mbox{diam}}{3}.$$
	\end{prop}
	\begin{proof}
		\ \\
		\normalfont
		Suppose $0<r< \frac{\mbox{diam}}{3}$. Then $\exists \hspace{0.1 cm} s_2 \in M \hspace{0.1 cm} \ni \hspace{0.1 cm} d(s_1,s_2) = 3r/2 $, for if otherwise $B(s_1, 3r/2) = \overline{B(s_1, 3r/2)}
		\neq M $ is clopen, which is a contradiction to the that $M$ is connected. Clearly, $B(s_1,r) \cap B(s_2, r/2) = 0$ and $B(s_2, r/2) \subset B(s_1, 2r)$, which yields $ |B(s_2, r/2)|+|B(s_1, r)|\leq |B(s_1, 2r)|$. In other word $B(s_1, r) \subset B(s_2, 5r/2)$ which along with (10) imply that $|B(s_1,r)| \leq 10^k B(s_2,r/2)$, thus $|B(s_1, 2r)| \geq (10{-k} +1)|B(s_1, r)| = 2^{\alpha} |B(s_1, r)|.$
	\end{proof}

	\section{Some Estimates}
	The symmetric function
	\begin{equation}
		E_{\delta, \varsigma} (s_1,s_2) := (|B(s_1,\delta)||B(s_2,\delta)|)^{-1/2} \left(1 + \frac{d(s_1,s_2)}{\delta}\right)^{-\varsigma}, \hspace{1 cm} s_1,s_2 \in M.
	\end{equation}
	will govern the localization of various operator kernels.
	Where $\delta, \varsigma>0$ represents the parameters which will be specified in each particular case.\\
	In what follows, we shall give some simple properties of $E_{\delta, \varsigma} (s_1,s_2)$ which will be very important instruments
	in various proofs later. We notice that \eqref{J} and \eqref{K} yields
	\begin{equation}
		E_{\delta, \varsigma} (s_1,s_2) \leq 2^{k/2} |B(s_1, \delta)|^{-1} \left(1 + \frac{d(s_1,s_2)}{\delta}\right)^{\varsigma-k/2},
	\end{equation}
	\begin{equation}
		E_{\beta \delta, \varsigma} (s_1,s_2) \leq (2/\beta)^k E_{\delta, \varsigma} (s_1,s_2), \hspace{1 cm} 0<\beta<1,
	\end{equation}
	\begin{equation}
		E_{\beta \delta, \varsigma} (s|_1,s_2) \leq \beta^{\varsigma} E_{\delta, \varsigma} (s_1,s_2), \hspace{1 cm} \beta >1.
	\end{equation}
	In addition, for $0<p<\infty$ and $\varsigma >k(1/2 + 1/p) $
	\begin{equation}\label{O}
		\lVert E_{\delta, \varsigma} (s_1,\cdotp)\rVert_p = \left (\int_M [E_{\delta, \varsigma} (s_1,s_2)]^p d \sigma(y) \right)^{1/p} \leq a(p)|B(s_1, \delta)|^{1/p -1},
	\end{equation}
	where $ a = \left( \frac{2^{kp/2}}{2^{-k} - 2^{-(\varsigma - k/2)p}} \right)^{1/p}$ is a decreasing function of $p$, and
	\begin{equation}\label{Q}
		\int_M E_{\delta, \varsigma}(s_1,v)E_{\delta, \varsigma}(v,s_2)d\sigma (v) \leq aE_{\delta, \varsigma}(s_1,s_2) \hspace{0.2 cm}  \mbox{if} \hspace{0.2 cm} \varsigma>2k
	\end{equation}
	with $ a = \frac{2^{\varsigma + k + 1}}{2^{-k} - 2^{k-\varsigma}}$.\\
	The estimates above follow from the lemma below which is important in the sequel.
	\begin{lem}
		\normalfont
		(a) Suppose $\varsigma>k$, then
		\begin{equation}\label{L}
			\int_M \left(1 + \delta^{-1} d(s_1,s_2)\right) ^{-\varsigma} d \sigma (s_2) \leq a_1 |B(s_1, \delta)|, \hspace{0.2 cm} s_1 \in M \hspace{0.1 cm} \left(a_1 = (2^{-k} - 2^{-\varsigma})^{-1} \right), \hspace{0.2 cm} \mbox{for} \hspace{0.1 cm} \delta>0
		\end{equation}
		(b) Suppose $\varsigma>k$, then for $s_1,s_2 \in M$ and $\delta>0$, then we have
		
		\begin{equation*}
			\int_{M} \frac{1}{\left(1 + \delta^{-1} d(s_1,v) \right)^{\varsigma} \left(1 + \delta^{-1} d(s_2,v) \right)^{\varsigma}} d\sigma(v)\leq 2^{\varsigma} a_1 \frac{|B(s_1,\delta)| + |B(s_2, \delta)|}{\left(1 + \delta^{-1} d(s_1,s_2) \right)^{\varsigma}}
		\end{equation*}
		\begin{equation}\label{M}
			\leq 2^{\varsigma} (2^k +1) a_1 \frac{|B(s_1,\delta)|}{\left(1 + \delta^{-1} d(s_1,s_2) \right)^{\varsigma -k}} \leq 2^{\varsigma}(2^k+1)a_1|B(s_1,\delta)|
		\end{equation}
		(c) If $\varsigma>2k$, for $\delta>0$ and $s_1,s_2 \in M$
		\begin{equation}\label{N}
			\int_{M} \frac{1}{|B(v,\delta)| \left(1 + \delta^{-1} d(s_1,v) \right)^{\varsigma} \left(1 + \delta^{-1} d(s_2,v) \right)^{\varsigma}} d\sigma(s_2) \leq \frac{a_2}{\left(1 + \delta^{-1} d(s_1,s_2) \right)^{\varsigma}},
		\end{equation}
		where $ a_2 = \frac{2^{\varsigma +k + 1}}{2^{-k} - 2^{k-\varsigma}}$.\\
	\end{lem}
	\begin{proof}
		\ \\
		\normalfont
		Let $ S_0 := \{s_2\in M: d(s_1,s_2)< \delta\} = B(s_1,\delta)$ and\\ $S_j := \{s_2\in M : 2^{j-1} \delta \leq d(s_1,\delta)<2^j\delta\} = B(s_1,2^j \delta)  B(s_1,2^{j-1} \delta)$, \hspace{0.2 cm} $j\geq 1$.
		Then apply \eqref{A} we have that
		\begin{equation*}
			\int_M \left(1 + \delta^{-1} d(s_1,s_2) \right)^{-\varsigma}d\sigma(s_2) = \sum_{j\geq 0} \displaystyle \int_{S_j} \left(1 + \delta^{-1} d(s_1,s_2) \right)^{-\varsigma}d\sigma(s_2)
		\end{equation*}
$\leq |B(s_1,\delta)|+ (2^k -1) \sum_{j\geq 0} \frac{|B(s_1, 2^j \delta)|}{(1+2^j)^{\varsigma}}$\\
$\leq |B(s_1,\delta)| \left(1+(2^k -1)\sum_{j\geq 0} \frac{2^{jk}}{(1+2^j)^\varsigma} \right) \leq \frac{|B(s_1, \delta)|}{2^{-k} - 2^{-\varsigma}},$\\
and this proves \eqref{L}\\
		To proof \eqref{M}, note that we can write\\
		\begin{equation*}
			\frac{1+\delta^{-1} d(s_1,s_2)}{(1+\delta^{-1} d(s_1,v))(1+\delta^{-1} d(s_2,v))} \leq \frac{1}{1+\delta^{-1} d(s_1,v)} + \frac{1}{1+\delta^{-1} d(s_2,v)}
		\end{equation*}
		for the triangle inequality, and thus
		\begin{equation}
			\frac{(1+\delta^{-1} d(s_1,s_2))^{\varsigma}}{(1+\delta^{-1} d(s_1,v))^{\varsigma} (1+\delta^{-1} d(s_2,v))^{\varsigma}} \leq \frac{2^{\varsigma}}{(1+\delta^{-1} d(s_1,v))^{\varsigma}} + \frac{2^{\varsigma}}{(1+\delta^{-1} d(s_2,v))^{\varsigma}}.
		\end{equation}
		Now integrating the above and applying \eqref{L} we have that
		\begin{equation*}
			\sum_{j\geq 0}\displaystyle\int_{s_j}\left( \frac{2^{\varsigma}}{(1+\delta^{-1} d(s_1,v))^{\varsigma}} + \frac{2^{\varsigma}}{(1+\delta^{-1} d(s_2,v))^{\varsigma}} \right) d\sigma (v)
		\end{equation*}
		\begin{equation*}
			\leq 2^{\varsigma}\sum_{j\geq 0}\displaystyle\int_{s_j}\frac{1}{(1+\delta^{-1} d(s_1,v))^{\varsigma}}+\frac{1}{(1+\delta^{-1} d(v,s_2))^{\varsigma}} d\sigma (v)
		\end{equation*}
		\begin{equation*}
			\leq 2^{\varsigma}\left(|B(s_1,\delta)|+(2^{2k}-1)\sum_{j\geq 0}\frac{|B(s_1,2^j \delta)|}{(1+2^j)^{\varsigma}}\right)
		\end{equation*}
		\begin{equation*}
			\leq 2^{\varsigma}\left(|B(s_1,\delta)|+(2^{2k}-1)\sum_{j\geq 0}\frac{2^{jk}(|B(s_1,\delta)|)}{(1+2^j)^{\varsigma}}\right)
		\end{equation*}
		\begin{equation*}
			\leq 2^{\varsigma}(2^k +1)\left[|B(s_1,\delta)|\left(1+(2^{k}-1)\sum_{j\geq 0}\frac{2^{jk}}{(1+2^j)^{\varsigma}}\right)\right]
		\end{equation*}
		\begin{equation*}
			\leq 2^{\varsigma}(2^k +1)\left(\frac{|B(s_1,\delta)|}{2^{-k}-2^{-\varsigma}}\right)
			\leq 2^{\varsigma}(2^k+1)a_1|B(s_1,\delta)|
		\end{equation*}
		which is \eqref{M}.\\
		To prove \eqref{N}, we use the inequality above and \eqref{K} and we have that
		\begin{equation*}
			\frac{(1+\delta^{-1} d(s_1,s_2))^{\varsigma}}{|B(v,\delta)|(1+\delta^{-1} d(s_1,v))^{\varsigma} (1+\delta^{-1} d(s_2,v))^{\varsigma}}
		\end{equation*}
		\begin{equation}
			\leq \frac{2^{\varsigma +k}}{|B(s_1,\delta)|(1+\delta^{-1} d(s_1,v))^{\varsigma -k}} + \frac{2^{\varsigma +k}}{|B(s_2,\delta)|(1+\delta^{-1} d(s_2,v))^{\varsigma -k}}.
		\end{equation}
		Integrating again and applying \eqref{L}, we have the claim \eqref{N}.
	\end{proof}
	\ \\
	We shall need maximal dirac-nets ($\delta$-nets) on $M$ to construct the decomposition systems in what follows\\
	\\
	\begin{defx}
 Let $\mathcal{K}$ be a subset of $M$ we call a $\mathcal{K}$ dirac-net on $M$ (written as $\delta$-net on $M$) where $\delta>0$ if the distance $d(\iota, \rho) \geq \delta \; \; \forall \; \; \iota, \rho \in \mathcal{K} $. The quantity $\mathcal{K}$ is said to be a maximal dirac-net on $M$ if we can not find $s \in M$ for which $d(s,\iota) \geq \delta \hspace{0.2 cm} \forall \hspace{0.2 cm} \iota,\in \mathcal{K} \hspace{0.2 cm} and \hspace{0.2 cm} s \notin \mathcal{K}$.
	\end{defx}
	The proposition below helps to give some properties of maximal $\delta$-nets.\\
	\begin{prop}\label{B} Let $(M,d,\sigma)$ be a measurable metric space satisfying the doubling condition and  $\delta>0$.\\
		(a) A maximal dirac-net ($\delta$-net) on $M$ exists all the time.\\
		(b) Suppose $\mathcal{K}$ is a maximal dirac-net on $M$, then
		\begin{equation}\label{PP}
			M =\bigcup_{\iota \in \mathcal{K}} B(\iota, \delta)\; \hbox{and} \; B(\iota, \delta /2) \cap B(\rho, \delta /2) =0 \; \hbox{if} \; \iota \neq \rho, \iota, \rho \in \mathcal{K}
		\end{equation}
		(c) Suppose $\mathcal{K}$ is a maximal dirac-net on $M$. Then $\mathcal{K}$ is finite or countable, and we can find a disjoint partition $\{P_{\iota}\}_{\iota \in \mathcal{K}}$ of $M$ consisting of measurable sets such that
		\begin{equation}
			B(\iota, \delta /2) \subset P_{\iota} \subset B(\iota, \delta), \hspace{0.2 cm} \iota \in \mathcal{K}.
		\end{equation}
	\end{prop}
	\begin{proof}
		\ \\
		\normalfont
(a)~~ First, note that a maximal dirac-net is a maximal set in the collection of all dirac-nets on $M$ with regards to the natural ordering of sets, and by Zorn's lemma a maximal dirac-net on $M$ exists.\\
		\\
(b)~~ follow immediately from the definition of maximal dirac-nets.\\
		\\
(c)~~ Fix $l \in M$, and observe that, for any positive integer $n > \delta$, we have by estimations \eqref{J} and \eqref{K} that $|B(l,n)| \leq a(n, \delta)|B(\iota, \delta /2)|$ for $\iota \in \mathcal{K} \cap B(,n)$, where $a(n, \delta)$ is a constant depending on $n$ and $\delta$. On other way round, using \eqref{PP} we have\\
		$$ \sum_{\iota \in \mathcal{K} \cap B(l,n)} |B(\iota, \delta /2)| \leq |B(l, 2n)| \leq 2^k |B(l,n)|.$$
		Thus, $ \sigma(\mathcal{K} \cap B(l, n)) \leq 2^k a(n, \delta) < \infty $, which implies that $ \mathcal{K} $ is finite or countable.\\
		By ordering the elements of $\mathcal{K}$ in a sequence: $ \mathcal{K}= \{\iota_1, \iota_2, \cdotp,\cdotp,\cdotp\}$. Now we shall define the set $P_{\iota}$ of $M$ inductively. By setting
		$$ P_{\iota_1} := B(\iota_1, \delta) \setminus \bigcup_{\rho \in \mathcal{K}, \rho \neq \iota_1} B(\rho, \delta/2)$$
		and suppose $ P_{\iota_1}, P_{\iota_2}, \cdot, \cdot, \cdot, P_{\iota_{j-1}} $ have been defined already, we can now set
		$$P_{\iota_j} := B(\iota_j, \delta) \setminus \left[\bigcup_{\nu \leq j-1} P_{\iota_{\nu}} \bigcup_{\rho \in \mathcal{K}, \rho \neq \iota_j} B( \rho, \delta/2)\right].$$
		One can see that the sets $ P_{\iota_1}, P_{\iota_2}, \cdot, \cdot, \cdot $ adorn the properties of the claim.
	\end{proof}

	\begin{thm}
		\normalfont
		Assume that $\mathcal{K}$ is a maximal dirac-net on $M$ and $\{P_{\iota} \}_{\iota \in \mathcal{K}}$ is a disjoint partition of $M$ (as in Proposition\eqref{B}). Then
		\begin{equation}\label{R}
			\sum_{\iota \in \mathcal{K}} |P|_{\iota} \left( 1+ \delta^{-1} d(s, \iota) \right)^{-k-1} \leq 2^{2k+2}|B(s, \delta)|
		\end{equation}
		and
		\begin{equation}\label{S}
			\sum_{\iota \in \mathcal{K}} \left( 1+ \delta^{-1} d(s, \iota) \right)^{-2k-1} \leq 2^{3k+2}.
		\end{equation}
		In addition, for any $\delta_{\star} \geq \delta$
		\begin{equation}\label{T}
			\sum_{\iota \in \mathcal{K}} \frac{|p_{\iota}|}{|B(\iota, \delta_{\star})|} \left( 1+ \delta_{\star}^{-1} d(s,\iota) \right)^{-2k-1} \leq 2^{3k+2}.
		\end{equation}
		If $ \varsigma \geq 2k + 1,$
		\begin{equation}\label{U}
			\sum_{\iota \in \mathcal{K}} |p_{\iota}E_{\delta_{\star},\varsigma}(s_1, \iota)E_{\delta_{\star},\varsigma}(s_2, \iota) \leq 2^{\varsigma + 3k + 3} E_{\delta_{\star},\varsigma}(s_1,s_2).
		\end{equation}
		Also for $ \varsigma \geq 2k + 1,$ then
		\begin{equation}\label{V}
			\sum_{\iota \in \mathcal{K}} \left( 1+ \delta^{-1} d(s_1, \iota) \right)^{- \varsigma} \left( 1+ \delta^{-1} d(s_2, \iota) \right)^{- \varsigma} \leq 2^{\varsigma + 2k + 3} \left( 1+ \delta^{-1} d(s_1,s_2) \right)^{- \varsigma}.
		\end{equation}
	\end{thm}
	\begin{proof}
		\normalfont
To prove \eqref{R}, we proceed as follows.\\
By \eqref{K} it can be observed that  $|B(s, \delta_{\star})| \leq 2^k \left(1+ \delta_{\star}^{-1} d(s, \iota)\right)^k |B(\iota, \delta_{\star})|$. On other way round, for $v \in P_{\iota} \subset B(\iota, \delta)$
		$$ 1 + \delta_{\star}^{-1} d(s,v) \leq 1 + \delta_{\star}^{-1} d(s,\iota) + \delta_{\star}^{-1} d(\iota,v) \leq 2 \left(1 + \delta_{\star}^{-1} d(s,\iota)\right). $$
		\begin{eqnarray*}
			|P_{\iota}|(1+\delta^{-1}d(s,\iota))^{-k-1} &\leq & 2^k|P_{\iota}|(1+\delta^{-1}d(s,\iota))^{-1}\\
			&\leq & 2^{k+1}\displaystyle\int_{P_{\iota}}(1+\delta^{-1}d(s,v))^{-1} d\sigma (v)
		\end{eqnarray*}
		which gives
\begin{eqnarray*}
\sum_{\iota \in k_j}|P_{\iota}|(1+\delta^{-1}d(s,\iota))^{-k-1}
		&\leq&  2^{k+1}\displaystyle\int_{P_{\iota}}(1+\delta^{-1}d(s,v))^{-1} d\sigma (v)\\
		&\leq&  2^{k+1} \frac{|B(s,\delta)|}{2^{-k-1}}
\end{eqnarray*}
		and \eqref{R} follows.\\
		Next, we shall prove \eqref{S}, in this manner.\\
		\begin{eqnarray*}
			(1+\delta^{-1}d(s,\iota))^{-2k-1} &\leq &2^k(1+\delta^{-1}d(s,\iota))^{-k-1}\\
			&\leq & 2^{2k+1}\displaystyle\int_{\iota} (1+\delta^{-1}d(s,v))^{-k-1}d \sigma (v)
		\end{eqnarray*}
		which implies that
		$$\sum_{\iota \in k_j}(1+\delta^{-1}d(s,\iota))^{-2k-1}\leq
		2^{2k+1}\displaystyle\int_{\iota} (1+\delta^{-1}d(s,v))^{-k-1} d\sigma(v) \leq 2^{2k+1}\frac{1}{2^{-k-1}}$$
		which establishes \eqref{S}.\\
		The proof of \eqref{T},\eqref{U} and \eqref{V} can be seen in (\cite{Coulhon1},\cite{Coulhon2}).
	\end{proof}
	\subsection{Integral (Kernel) Operators}
	We will allow quantities $E_{\delta_{\star}, \varsigma}(s_1,s_2)$ (c.f \eqref{W}) to control the kernels of many operators. We shall need Young-type inequality for such operators.\\
	\begin{prop}\label{C}
		\normalfont
		Suppose $\mathscr{H}$ is an integral operator with kernel $\mathscr{H} (s_1,s_2)$, that is
		$$ \mathscr{H}f(s_2) = \int_M \mathscr{H}(s_1,s_2) f(s_1) d \sigma (s_1), \;\;\; \mbox{and}\; \mbox{if} \;\;\; |\mathscr{H} (s_1,s_2) | \leq a^{\prime} E_{\delta, \varsigma} (s_1,s_2) $$
		for some $1\geq \delta >0$ and $\varsigma \geq 2k+1$. If $\infty \geq q \geq p \geq 1$, then
		\begin{equation}\label{S}
			\lVert \mathscr{H}f \rVert_q \leq a\delta^{k(\frac{1}{q} - \frac{1}{p})} \lVert f \rVert_p, \hspace{1 cm} f \in \mathbb{L}^p,
		\end{equation}
		here $a=a^{\prime} \hat{a}^{k(1/r-1)} 2^{2k+1}$ with $\hat{a}$ being the constant from \eqref{Y}.
	\end{prop}
	The Lemma below affirms the above result.
	\begin{lem}\label{D}
		\normalfont
		Assume $ \frac{1}{p} - \frac{1}{q} = 1 - \frac{1}{r}$ where $1\leq p,q,r\leq \infty$, and given the measurable kernel $ \mathscr{H} (s_1,s_2)$, if we have the conditions that
		\begin{equation*}
			\lVert \mathscr{H}(\cdot,s_2) \rVert_r \leq C \hspace{1.5 cm} \mbox{and} \hspace{1.5 cm} \lVert \mathscr{H}(s_1, \cdot) \rVert_r \leq C,
		\end{equation*}
		suppose $\mathscr{H}f(s_1) = \displaystyle\int_M \mathscr{H}(s_1,s_2)f(s_2)d\sigma(s_2)$ then
		$$ \lVert \mathscr{H}f \rVert_q \leq C \lVert f \rVert_p, \hspace{1 cm} f \in \mathbb{L}^p.$$\hfill{$\Box$}
	\end{lem}
	\begin{proof}{\textbf{Proof of Proposition}\eqref{C}}\\
		\normalfont
		Choose $1\leq r\leq \infty$ such that $\frac{1}{p} - \frac{1}{q} = 1 - \frac{1}{r}$. Using \eqref{O} and \eqref{Y} we obtain
		$$ \lVert \mathscr{H}(s_1, \cdot) \rVert_r \leq a^{\prime} a(r)|B(\delta, s_1)|^{1/r -1} \leq a^{\prime} a(1) (\hat{a} \delta)^{k(1/r -1)} $$
		$\lVert \mathscr{H}(\cdot, s_2) \rVert_r $ follow in the same pattern. Using the estimation above coupled with Lemma\eqref{D} we have that.\\
		From the condition of the proposition\\
		\begin{eqnarray*}
			\lVert \mathscr{H}f\rVert_q &\leq & a^{\prime} E_{\delta,\varsigma}(s_1,s_2)\lVert f \rVert_p\\
			& \leq & a^{\prime}a(q)|B(\delta,s_1)|^{\frac{1}{q}-\frac{1}{p}} \lVert f \rVert_p\\
			& \leq & a^{\prime}(\hat{a}\delta)^{k(\frac{1}{r}-1)} 2^{2k+1} \lVert f \rVert_p\\
			& \leq & a\delta^{k(\frac{1}{r}-1)} \lVert f \rVert_p, \;\; f\in \mathbb{L}^p
		\end{eqnarray*}
	\end{proof}

\end{document}